\documentclass[10pt, reqno]{amsart}
 \usepackage{tikz}
 \usepackage{hyperref, mathtools, bm}
 \usepackage[skip=4pt]{caption}
 \usepackage{amssymb,amsmath,graphicx,amsthm}
  \usetikzlibrary{positioning}
  \usepackage{tkz-euclide}
\usepackage{rotating}
\usepackage{color,xcolor}
\definecolor{darkred}{rgb}{1,0,0} 
\definecolor{darkgreen}{rgb}{0,0.8,0}
\definecolor{darkblue}{rgb}{0,0,1}

\hypersetup{colorlinks,
linkcolor=darkblue,
filecolor=darkgreen,
urlcolor=darkred,
citecolor=darkgreen}
 \usepgflibrary{patterns}

 \def\bt{\begin{theorem}}
 	\def\el{\end{lemma}}
 \def\bl{\begin{lemma}}
 	\def\et{\end{theorem}}
 \def\bp{\begin{proposition}}
 	\def\ep{\end{proposition}}
 \def\bd{\begin{definition}}
 	\def\ed{\end{definition}}
 \def\br{\begin{remark}}
 	\def\er{\end{remark}}

 \newcommand{\Om}{\Omega}

 \def\R{{\mathbb R}}
 
 \def\I{{\mathcal I}}
 
 \def\C{\mathbb C}
 \def\B{\mathbb B}
 
 \def\N{{\mathbb N}}
 
\def\S{{\mathbb S}}
\def\S1{{\mathbb S^1}}
\def\H{{\mathbb H}} 
\def\i{{\bold i}}
\def\j{{\bold j}}
\def\k{{\bold k}}

 \def\label#1{\label{#1}{\bf(#1)}~}

 \numberwithin{equation}{section}

 \theoremstyle{plain} 

 \theoremstyle{plain}
 \newtheorem*{theorem*}{Theorem}

 \newtheorem{theorem}{Theorem}[section]

 \newtheorem{lemma}[theorem]{Lemma}
 
 \newtheorem{proposition}[theorem]{Proposition}

 \theoremstyle{definition}
 
 \newtheorem{definition}[theorem]{Definition}
 
 \theoremstyle{remark}

 \newtheorem{remark}[theorem]{Remark}

 \newcommand{\p}{\partial}

 \newcommand{\dbar}{\bar\partial}

 \DeclareMathOperator{\dist}{dist}

  \usetikzlibrary{calc,intersections,through,backgrounds}
 
 \def\Re{\text{Re}}

\begin{document}
 	
 	\title[Extension of regular functions of two quaternionic variables]{An extension theorem for regular functions of two quaternionic variables}  
\author{Luca Baracco}
\address{Dipartimento di Matematica Tullio Levi-Civita, Universit\`a di Padova, via Trieste 63, 35121 Padova, Italy}
\email{baracco@math.unipd.it}

\author{Martino Fassina}
\address{Department of Mathematics, University of Illinois, 1409 West Green
Street, Urbana, IL 61801, USA}
\email{fassina2@illinois.edu}

\author{Stefano Pinton}
\email{spinton84@gmail.com}
 	

 \begin{abstract} 
For functions of two quaternionic variables that are regular in the sense of Fueter, we establish a result similar in spirit to the Hanges and Tr\`eves theorem. Namely, we show that a ball contained in the boundary of a domain is a propagator of regular extendability across the boundary. 
\end{abstract}
  	\maketitle
  	\maketitle

\section{Introduction and Main Results}

The algebra $\H$ of quaternions was introduced in 1843 by William R. Hamilton in the attempt to extend the system of complex numbers. Recall that a quaternion $p\in\H$ can be written as
\begin{equation}\label{p}
p=x_0+x_1\i +x_2\j+x_3\k,\quad x_i\in\R,
\end{equation}
where $\i,\j,\k$ are imaginary units ($\i^2=\j^2=\k^2=-1$) satisfying the following multiplicative relations:
$$ \i\j=-\j\i=\k,\quad \j\k=-\k\j=\i,\quad \k\i=-\i\k=\j.$$
It is natural to ask if there exists a class of functions on quaternions that replicates the classical theory of holomorphic functions in 
one and several complex variables. One approach to this question is to define the so called Cauchy-Fueter operator 
\begin{equation*}\label{Futt}
 \dbar := \p_{x_0} + \i \p_{x_1}+\j\p_{x_2}+\k\p_{x_3},
\end{equation*}
and declare a function to be (left) regular if it satisfies $\dbar f=0$, in analogy with the usual Cauchy-Riemann equations. This formulation is due to Moisil \cite{M31}. Fueter and his students then gave a great impulse to the development of the theory of regular functions \cite{F1,F2,Ha}. (See also the survey paper \cite{S}). 

Fueter and his school were able to obtain a quaternionic counterpart for regular functions of many classical theorems of complex analysis in one variable such as Cauchy's theorem, 
Liouville's theorem and Laurent series expansion. (See \cite{D} for a survey of these results). Since then, 
the theory of regular functions has been greatly developed, and the parallel with holomorphic functions has been extended to several variables. For instance, Pertici \cite{P} established a Bochner-Martinelli integral formula for regular functions in $\H^n$, and used it to prove an analog of Hartogs extension theorem in the quaternionic setting. A purely algebraic proof of Hartogs extension for regular functions appears in \cite{CSSS}.


In this paper we consider (left) regular functions in the sense described above. We pursue the parallel with holomorphic functions 
by proving an analog of two classical results in the theory of several 
complex variables. 

A famous theorem of Hartogs states that a function is holomorphic in $\C^n$ if and only if it satisfies the Cauchy-Riemann equations in each variable separately (\cite[Theorem 2.2.8]{H}). 
In other words, no joint regularity whatsoever is required to ensure analyticity. 
We observe that a similar result holds for regular functions in
$\mathbb{H}^n$.

\begin{proposition}\emph{(Proposition \ref{main})}\label{t1}
Let $\Omega$ be an open domain in $\H^n$ and $f\colon\Omega\rightarrow \H$ a function satisfying the Cauchy-Fueter equations in each variable separately. Then $f$ is regular in $\Omega$. 
\end{proposition}

Proposition \ref{t1} follows easily from a theorem of Avanissian \cite[Theorem 1]{A67} on separately harmonic functions (see Remark \ref{easyproof}). We decided to present a detailed proof in order to introduce the complexification techniques from \cite{A67}, which play a crucial role in the proof of the main result.
 
The main theorem of the paper extends to the quaternionic setting a well-known general principle in complex analysis: complex curves in a hypersurface $M$ are propagators of holomorphic extendability from either side of $M$. (See \cite[Chapter 1.7]{Z}).  
The main result in this direction is due to Hanges and Tr\`eves \cite{HT}, and was proved using microlocal techniques. See the papers \cite{B05} and \cite{MSZ} for other instances of this phenomenon. Here is our result.

\begin{theorem}\emph{(Theorem \ref{hangestreves})}\label{t2} Let $\Omega\subset \H^2$ be a domain with $C^2$ boundary $\partial\Omega$. Let $\B_1$ denote the unit ball in $\H$ centered at the origin, and assume that $\gamma:=\B_1\times \{0\}$ is contained in $\p \Omega$.
Let $U$ be a neighborhood of $(1,0)$ in $\H^2$,
and suppose that $f$ is a regular function on $\Omega\cup U$. Then there exists a neighborhood $W$ of $\gamma$ in $\H^2$ and a regular function $F$ on $W$ such that $F_{|W\cap \Omega}=f_{|W\cap \Omega}$.
That is, there exists a regular extension of $f$ to a neighborhood of $\gamma$.
\end{theorem}

Figure \ref{fig1} gives a simple illustration of Theorem \ref{t2}. One starts with a function $f$ that is regular in the grey region $\Omega\cup U$. The ball $\gamma$ in the boundary $\p\Omega$ propagates the regular extension of $f$ across $\p\Omega$ to a full neighborhood $W$ of $\gamma$. 

\begin{figure}[h]
\centering
 \begin{tikzpicture}[scale=0.8]
                        
\draw[black,thin,fill=gray, fill opacity = 0.1] (2.6,0) circle (1.5cm);

\draw 
[line width=0.6mm, black ] (-2.6,0) -- (2.6,0);
\path [line width=0.2mm, black, fill=gray, fill opacity = 0.1 ](2,0)..controls (3,0) .. (5,-2)--(-6,-2)..controls (-3,0) .. (-2,0)--(2,0);
\draw [line width=0.001mm, black](-2.6,0)..controls(0,-0.15)..(2.6,-0.5)--(2.6,0.5)..controls(0,0.15)..(-2.6,0);
\path[pattern=north west lines, pattern color=blue] (-2.6,0)..controls(0,-0.15)..(2.6,-0.5)--(2.6,0.5)..controls(0,0.15)..(-2.6,0);
\draw [line width=0.2mm, black](2,0)..controls (3,0) .. (5,-2);
\draw [line width=0.2mm, black](-6,-2)..controls (-3,0) .. (-2,0)--(2,0);
\draw[fill] (2.6,0) circle [radius=0.1] ; 
\node at (3.1,0.3) {$(1,0)$};




\node at (-2,0.3) {$\bm{\gamma}$};
\node at (0,-1.3) {$\Omega$};
\node at (-4.3,-0.4) {$\p\Omega$};
\node at (3.5,1.5) {$U$};
\node at (0,0.7) {$W$};



\end{tikzpicture}
\caption{} \label{fig1}
\end{figure}

The paper is organized as follows. 
In Section \ref{Section 1} we give a self-contained exposition of all the results that we need from the theory of regular functions in one quaternionic variable. Most of the material in this section is well known.
In Section \ref{Section 2} and Section \ref{Section 3} we present the proofs of Proposition \ref{t1} and Theorem \ref{t2} respectively.

\section{Regular functions of one quaternionic variable}\label{Section 1}
We start by recalling some definitions and setting the notation. 
For an element $p\in\H$ as in \eqref{p}, we define the conjugate $\bar{p}:=x_0-x_1\i -x_2\j-x_3\k$. We then let $$|p|:=\sqrt{p\bar{p}}=\sqrt{x_0^2+x_1^2+x_2^2+x_3^2}.$$
Note that if $p\in\H, p\neq 0$, then $p^{-1}=\bar{p}/|p|^2$.

We now introduce the Cauchy-Fueter operators 
$$ \dbar := \p_{x_0} + \i \p_{x_1}+\j\p_{x_2}+\k\p_{x_3} \quad\,\,\, \p:=\p_{x_0} - \i \p_{x_1}-\j\p_{x_2}-\k\p_{x_3}$$
$$ \dbar^R := \p_{x_0} +  \p_{x_1}\i+\p_{x_2}\j+\p_{x_3}\k \quad \,\,\,\p^R:=\p_{x_0} - \p_{x_1}\i-\p_{x_2}\j-\p_{x_3}\k.$$

\bd 
Let $\Omega\subset\H$ be an open set. A $C^1$ function $f\colon \Omega\rightarrow\H$ is said to be {\em left regular} if $\bar\partial f=0$ in $\Omega$ and {\em right regular} if $\bar\partial^R f=0$ in $\Omega$.
\ed
Throughout the paper we will deal with left regular functions. We will thus use the term {\em regular function} to mean a left regular function. Obviously, a similar theory can be formulated for right regular functions. 
 
As stated in the next theorem, regular functions satisfy an analog of the Cauchy integral representation formula. 
\bt\label{Fueterthm} Let $\Omega\subset\H$ be an open set and $f\colon\Omega \rightarrow \H$ a left regular function. Let $\Omega'$ be an open bounded set with smooth boundary such that $\overline{\Omega'}\subset\Omega$. 
Then, for every $p_0\in \Omega'$, we have
\begin{equation}\label{CF}
f(p_0)=\frac 1{2\pi^2}\int_{\partial \Omega'} \frac{(p-p_0)^{-1}}{|p-p_0|^2} D(p) f(p),
\end{equation}
where $D(p):=dx_1\wedge dx_2\wedge dx_3 -\i dx_0\wedge dx_2 \wedge dx_3 +\j dx_0\wedge dx_1 \wedge dx_3-\k dx_0 \wedge dx_1 \wedge dx_2 $.
\et

It follows from the Cauchy-Fueter integral formula \eqref{CF} that regular functions are real analytic and satisfy the maximum principle. The proofs are similar to those of the corresponding statements for 
holomorphic functions of one complex variable. 

The next lemma shows another important consequence of the integral representation formula \eqref{CF}: a bounded regular function admits an estimate for its partial derivatives, uniform on compact sets.
\begin{lemma}\label{Derivative}
 Let $\Omega\subset\H$ be an open set and $f\colon\Omega\rightarrow\H$ a left regular function. Assume that $|f|\leq M$ on $\overline{\Omega}$ for some constant $M\geq 0$. Then, for each $p_0\in\Omega$ and $i=0,1,2,3$, we have
\begin{equation*}
 |\partial_{x_i}f(p_0)|\leq C\,\frac{M}{\dist(p_0,\partial\Omega)},
\end{equation*}
where $C$ is a universal constant independent of $f$ and $\Omega$.
\end{lemma}
 


As we noted above, regular functions are real analytic. Unlike the case of holomorphic functions, however, the power series of a regular function does not have a natural domain of convergence. 
The following lemma gives a lower bound for the radius of convergence of the Taylor series at $0$ of a regular function on a ball. 

\begin{lemma}\label{converg}
Let $\B_{R}\subset\H$ be the ball of radius $R$ centered at $0$ and $f\colon\B_R\rightarrow\H$ a regular function. Then the Taylor series of $f$ at $0$ converges to $f$ on the ball of radius $R(\sqrt{2}-1)$ 
centered at 
the origin.
\end{lemma}
\begin{proof}
It is not restrictive to assume that $f$ is continuous up to the boundary $\partial\B_R$. The Cauchy-Fueter formula \eqref{CF} implies, for every $p_0\in\B_R$, that
\begin{equation}\label{into}
 f(p_0)=\frac 1{2\pi^2}\int_{\p \B_R}\frac{\overline{p-p_0}}{|p-p_0|^4} D(p)f(p).
\end{equation}
We rewrite the kernel as
$$ \frac{\overline{p-p_0}}{|p-p_0|^4}=\frac{\overline{p}}{|p|^4} (1-\overline{p^{-1}p_0})\frac 1{(|1-p^{-1}p_0|^2)^2} .$$
Applying the identity $$\frac 1{(1-x)^2}=\sum_{n=1}^{+\infty} nx^{n-1}, \,\text{ with }\, x=2\Re (p^{-1}p_0 ) -|p^{-1}p_0|^2,$$ we obtain
\begin{equation}\label{uniform}
\begin{split}
\frac{\overline{p-p_0}}{|p-p_0|^4}&=\frac{\bar p}{|p|^4}(1-\overline{p^{-1}p_0})\sum_{n=1}^{+\infty} n (-2\Re (p^{-1}p_0 ) +|p^{-1}p_0|^2)^{n-1}\\ 
&=\frac{\bar p}{|p|^4}(1-\overline{p^{-1}p_0})\sum_{n=1}^{+\infty} n \sum^{n-1}_{k=0} (|p^{-1}p_0|^2)^k (-2\Re(p^{-1}p_0))^{n-1-k} \binom{n-1}{k} .
\end{split}
\end{equation}
Note that the series in \eqref{uniform} is absolutely convergent for $|p^{-1}p_0|^2+2|\Re(p^{-1}p_0)| < 1$.
The most restrictive situation is when $p^{-1}p_0$ is real, in which case we have
$$ |p^{-1}p_0|^2+2|(p^{-1}p_0)| < 1 \, \Longleftrightarrow\, |(p^{-1}p_0)|<\sqrt{2}-1 .$$
If we therefore assume $ |p_0| <(\sqrt{2} -1)R$, then we can group the terms with homogeneus powers of $p_0$, substitute \eqref{uniform} into \eqref{into} and integrate term by term. This concludes the proof.
\end{proof}

\begin{remark} The region of convergence of the power series of a regular function is a widely studied subject \cite{GHS}, and better bounds than the one proved in Lemma \ref{converg} are available in the literature. We decided to include Lemma \ref{converg} because the proof is elementary and self-contained, and the estimate obtained is sufficient for our purposes.  
\end{remark}

We recall a standard fact about power series.
\bp\label{cube} Let $\I^n_l:=\prod^n_{i=1}(-l,l) \subset \R^n$ be the $n$-dimensional cube of side $2l$ centered at $0$. Let $\sum_{\alpha\in \N^n} a_\alpha x^\alpha$ be a power series in $\R^n$. Let $|\alpha|=\sum_{j=1}^n\alpha_j$ and 
$$ c:=\limsup_{|\alpha|\to +\infty} |a_\alpha|^{\frac 1{|\alpha|}} \in [0,+\infty].$$
Then the series $\sum_{\alpha\in \N^n} a_\alpha x^\alpha$ conveges on $\I^n_l$ if and only if $1 \ge lc$.
\ep
\begin{remark}\label{rem} Assume that the regular function $f$ is uniformly bounded in $\B_R$ by a positive constant $M$. For $\alpha\in\mathbb{N}^n$ let $\partial^{\alpha}=\partial^{\alpha_1}_{x_1}\partial^{\alpha_2}_{x_2}\partial^{\alpha_3}_{x_3}\partial^{\alpha_4}_{x_4}$. It follows from Lemma \ref{converg} and Proposition \ref{cube} that 
\begin{equation} \label{highder}\left| \frac {\partial^\alpha f(0)}{\alpha !}\right|\le \frac {2^{|\alpha|}CM}{((\sqrt{2}-1)R)^{|\alpha |}} .\end{equation} 
for some constant $C$ independent of $\alpha$ and $f$.
\end{remark}
 
 We will need the following theorem of Lelong.
\begin{theorem}\cite[Theorem 10]{L}\label{Ll} Let $\Omega\subset \C^n$ be an open connected domain and let $D=\Omega\cap \R^n$ denote its intersection with the real subspace. Let $u_n:\Omega \rightarrow \R$ be a sequence of plurisubharmonic functions that are locally uniformly bounded from above. Assume that there exists a continuous function $g:\Omega\rightarrow \R$ such that 
$$ \limsup_{n\to +\infty}\,u_n(x) \le g(x)\quad  \forall x\in D.$$
Then for every relatively compact subset $K\subset\subset D$ and every $\epsilon >0$, there exists a positive integer $n_{K,\epsilon}$ such that, for every  $n\ge n_{K,\epsilon}$, 
$$ u_n(x) \le g(x)+\epsilon \quad \forall x\in K.$$ 
\end{theorem}

\section{Regular functions of several quaternionic variables}\label{Section 2}
We turn our attention to functions of several quaternionic variables. We denote by $p_1,\dots,p_n$ the coordinates in $\H^n$ and write $\bar\partial_{p_i}$ for the Cauchy-Fueter derivative in the 
direction of the variable $p_i$.

\bd Let $\Omega\subset\H^n$ be an open set. We say that a function $f\colon\Omega\rightarrow\H$ is {\em (left) regular} if $f\in C^1(\Omega)$ and it satisfies the Cauchy-Fueter equations in all variables, that is, $\bar\partial_{p_i}f=0$ in $\Omega$ for $i=1,\dots,n.$
\ed

Exploiting the Cauchy-Fueter integral formula of Theorem \ref{Fueterthm}, it is easy to show that regular functions are real analytic.
In this section we show that real analyticity follows even if the hypothesis $f\in C^1(\Omega)$ is dropped. We thus formulate this {\em a priori} weaker definition.
\bd
Let $\Omega\subset\H^n$ be an open set. We say that a function $f\colon\Omega\rightarrow\H$ is {\em separately (left) regular} if $f$ is regular in each variable $p_j$ when the other variables are given arbitrary fixed values.
\ed
It is obvious that a regular function on an open domain $\Omega\subset\H^n$ is also separately regular. In the next proposition we establish the converse.

\begin{proposition}\label{main}
Let $\Omega$ be an open domain in $\H^n$ and $f\colon\Omega\rightarrow \H$ a separately regular function. Then $f$ is regular in $\Omega$. 
\end{proposition}
\begin{remark}\label{easyproof} 
Proposition \ref{main} can be easily proved by first observing that the components of a regular function of a quaternionic variable are harmonic, and then applying \cite[Theorem 1]{A67}, which states that separately harmonic functions are harmonic. We present a full proof below, mainly following the reasoning of \cite{A67}, with some simplifications given by the Cauchy-Fueter integral formula. Spelling out the argument in full gives us an opportunity to introduce the complexification techniques that will be also used in the proof of Theorem \ref{hangestreves}.
\end{remark}
We start the proof of Proposition \ref{main} by showing that if $f$ is also assumed to be continuous in $\Omega$, then the result follows easily.
\begin{lemma}\label{lemmacont}
Let $\Omega\subset\H^n$ be an open domain and $f\colon \Omega\rightarrow\H$ a separately regular function which is also continuous in $\Om$. Then $f$ is regular in $\Om$.
\end{lemma}
\begin{proof}
The only thing to prove is that $f\in C^1(\Om)$. We argue just for the case $n=2$, since the general case follows from iterating the same argument. Hence assume $\Om\subset\H^2$, and let $(p_0,q_0)\in\Omega$. Recall that we denote by $\B_{r}(p_0)\subset\H$ a ball of radius $r$ centered at $p_0$. Consider now the product of two balls $\B_{r_1}(p_0)\times\B_{r_2}(q_0)\subset\Om$. The Cauchy-Fueter integral formula \eqref{CF} implies
\begin{equation*}
\begin{split}
 f(p_0,q_0)&=\frac 1{2\pi^2} \int_{\p\B_{r_1}(p_0)}\frac{(p-p_0)^{-1}}{|p-p_0|^2}\,D(p)\,f(p,q_0)\\ &=\frac 1{2\pi^2} \,\int_{\p\B_{r_1}(p_0)}\frac{(p-p_0)^{-1}}{|p-p_0|^2}\,D(p)\bigg( 
\frac 1{2\pi^2} \int_{\p\B_{r_2}(q_0)}\frac{(q-q_0)^{-1}}{|q-q_0|^2}\,D(q)\,f(p,q)\bigg)\\ &=\frac{1}{4\pi^4}\int_{\p\B_{r_1}(p_0)\times \p\B_{r_2}(q_0)} \frac{(p-p_0)^{-1}}{|p-p_0|^2}\,D(p)\,\frac{(q-q_0)^{-1}}{|q-q_0|^2}\,D(q)\,f(p,q), 
\end{split}
\end{equation*}
 where the last equality follows from Fubini's theorem. Hence $f$ is of class $C^1$ (and actually real analytic) in $\Omega$. \end{proof}
 
The next lemma shows that continuity for a separately regular function $f$ is a consequence of being uniformly bounded on compact sets. 
\begin{lemma}\label{lemmaboun}
Let $\Omega\subset\H^n$ be an open domain and $f\colon \Omega\rightarrow\H$ a separately regular function which is uniformly bounded on compact subsets of $\Om$. Then $f$ is regular in $\Om$. 
\end{lemma}
\begin{proof}
Lemma \ref{Derivative} implies that $f$ is uniformly Lipschitz separately in the variables $p_j$ with the same Lipschitz constant for all $j=1,\dots,n$. Hence $f$ is jointly continuous in the variables $p_j$, and therefore regular by Lemma \ref{lemmacont}.
\end{proof}

\begin{proof}[Proof of Proposition \ref{main}]
We will only prove the statement for regular functions of two quaternionic variables, since the general case then follows by iteration. It is also not restrictive to assume that the domain of the function $f$ is the product $\B_1\times\B_1\subset\H^2$ of two unitary balls centered at the origin. Additionally, we can take $f$ to be defined up to the boundary. For each $n\in\mathbb{N},$ let 
$$ E_n:=\{ q\in\B\,\colon\,|f(p,q)| \le n \,\,\,\forall p\in \overline{\B}_1 \}.$$
Clearly $E_n\subset E_{n+1}$ for every $n\in\mathbb{N}$. Moreover, 
$$ \bigcup_{n\in\N}E_n =\overline{\B}_1 .$$
The continuity of $f(p,\cdot)$ for fixed $p$ implies that if $q_k\to q$, with $q_k\in E_n$ for all $k$, then $q\in E_n$. Thus $E_n$ is closed for every $n$. By the Baire category theorem, there exists $M$ such that $E_M$ has non-empty interior. Assume that $0$ is in the interior of $E_M$. Then $f$ is uniformly bounded on $\B_1\times \B_\epsilon$ for some $\epsilon>0$. Indeed \[ |f(p,q)|\leq M\quad \forall (p,q)\in \B_1\times \B_\epsilon. \]
It follows from Lemma \ref{lemmaboun} that $f$ is regular and hence analytic on $\B_1\times \B_\epsilon$. At this stage, we can even forget that $f$ is separately regular in $p$ outside the strip $\B_1\times \B_\epsilon$. We will exploit the separate regularity in the variable $q$ to prove that $f$ is regular in $\B_1\times \B_1$.

Consider the Taylor series of $f$ at $0$ with respect to the variable $q$:
\begin{equation} \label{equazione5} f(p,q)= \sum_{\alpha\in \N^4} \frac{\p^\alpha f(p,0)}{\alpha !}\,q^\alpha. \end{equation}
We have used the notation $q^\alpha(p):=y_0^{\alpha_0}y_1^{\alpha_1}y_2^{\alpha_2}y_3^{\alpha_3}$, $\p^\alpha(p):=\p_{y_0}^{\alpha_0}\p_{y_1}^{\alpha_1}\p_{y_2}^{\alpha_2}\p_{y_3}^{\alpha_3}$, where the $y_j$ are real coordinates for $q$, that is, $q=y_0+y_1\i+y_2\j+y_3\k$. 

For every multi-index $\alpha$, the function \[g_{\alpha}(p):=\frac{\p^\alpha f(p,0)}{\alpha !}\] is regular in $p=x_0+x_1\i +x_2\j +x_3\k$.
By Remark \ref{rem} we have 
\begin{equation}\label{pipp}
|g_\alpha (p)|\le \frac {2^{|\alpha|}C M}{ ((\sqrt{2}-1)\epsilon)^{|\alpha|}}\quad \forall p\in \mathbb{B}_1.
\end{equation}
Since $g_\alpha$ is left regular, the Cauchy-Fueter formula implies
\begin{equation}\label{pq}
g_\alpha(p)= \frac 1{2\pi^2}\int_{\p \B_R}\frac{ (\xi_0-x_0)-(\xi_1-x_1)\i-(\xi_2-x_2)\j-(\xi_3-x_3)\k } {( (\xi_0-x_0)^2+(\xi_1-x_1)^2+(\xi_2-x_2)^2 +(\xi_3-x_3)^2 )^2} D(\zeta)g_{\alpha}(\zeta),
\end{equation}
where $\xi=\xi_0+\xi_1\i+\xi_2\j+\xi_3\k$. Since the kernel in the integral is a rational function, we can follow the method of Avanissian \cite[Section 3]{A67} and consider a complexification of $g_{\alpha}$ by taking  $x_j\in \C$ in equation \eqref{pq}.
We thus obtain a function $\tilde{g}_{\alpha}$ of the complexified variable $p^{\C}$ with values in $\C\otimes \H$ defined on an open subset $\Gamma$ of $\C\otimes\H$. Under the natural identification of $\C\otimes\H$ with $\C^4$, we have
$$ 
\Gamma =\Big{\{} x=(x_0,x_1,x_2,x_3)\in \C^4\, \big{|}\, \sum_{l=0}^3(\xi_l-x_l)^2\neq 0 \text{ for all } \xi \in \p \B_1 \Big{\}}.
$$  
Consider the connected component of $\Gamma$ that contains $\B_1$ and call it again $\Gamma$. The fact that $g_{\alpha}$ is uniformly bounded on $\B_1$ propagates to the complexification. Indeed, exploiting the integral formula \eqref{pq} and the estimate \eqref{pipp}, we see that for every compact set $K\subset \Gamma$ there exists a constant $\tau_K >0$ such that  
\begin{equation}\label{bdalfa}
 \sup_{K}{|\tilde{g}_\alpha|} \le \frac {2^{|\alpha|}C M}{ ((\sqrt{2}-1)\epsilon)^{|\alpha|}}(1+\tau_K).
\end{equation}   
For every multi-index $\alpha$, consider the function $u_{\alpha}$ defined on $\Gamma$ by
$$ u_\alpha (x) :=\log \left| \tilde{g}_{\alpha}(x)\right |^{\frac 1{|\alpha|}} .$$
The functions $u_\alpha$ are plurisubharmonic and by \eqref{bdalfa} they are locally uniformly bounded from above. 
Since $f$ is regular in the variable $q$ on the ball $\B_1$, Lemma \ref{converg} implies that the series  \eqref{equazione5} converges on the ball of radius $\sqrt{2}-1$, and in particular on the cube $\I_{\frac{\sqrt{2}-1}{2}}$ of side $\sqrt{2}-1$ centered at $0$. By Proposition \ref{cube}, we have
$$ \limsup_{|\alpha|\to +\infty} u_\alpha(x) \le \log\bigg(\frac{2}{\sqrt{2}-1}\bigg) .$$
Theorem \ref{Ll} then implies that the series \eqref{equazione5} converges uniformly on $\B_1\times \I_{\frac{\sqrt{2}-1}{2}}$, and so $f$ is smooth there. By repeating this argument, one can prove that $f$ is smooth everywhere in $\B_1\times\B_1$, and therefore regular there.
\end{proof}

\section{Propagation of regular extension across the boundary}\label{Section 3}
We begin this section by recalling some notation. We denote by $\B_{R}(p)$ a ball of radius $R$ in $\H$ centered at a point $p\in\H$, and by $\B^{\C}_{R}(p^{\C})$ a ball of radius $R$ in $\C\otimes \H$ centered at a point $p^{\C}\in\C\otimes\H$. Moreover, we write $\I_R(p)$ for the cube $\prod_{j=0}^3(x_j-R,x_j+R)$ in $\H$ of side $2R$ centered at $p=x_0+x_1\i +x_2\j+x_3\k\in\H$. If the center (either of the ball or the cube) is not specified, then it is assumed to be the origin.

\bd  Let $\Omega\subset\H^2$ be an open set and let $f$ be a regular function on $\Omega$. For every point $(p_0,q_0)\in\Omega$ we define the quantity 
\begin{equation*}
r_q(p_0,q_0):=\left( \limsup_{|\alpha|\to +\infty} \left| \frac{(\p^\alpha f (p_0,q))_{|q=q_0}}{\alpha !}\right|^{\frac 1 {|\alpha|}} \right)^{-1}. 
\end{equation*}
\ed
Note that $r_q(p_0,q_0)$ measures the convergence of the Taylor series of $f$ at $(p_0,q_0)$ with respect to the variable $q$. Such series converges for $q$ varying in the cube $\I_{r_q(p_0,q_0)}(q_0)$. If $p_0\times \B_R(q_0) \subset \Omega$, then Lemma \ref{converg} implies $r_q(p_0,q_0)\ge R(\sqrt{2} -1)$. 

\bp \label{proposizione4} Let $f$ be a regular function on $\B_1\times\B_\epsilon$, and suppose that $$r_q(p,0) \ge R\text{ for all }p\in \B_1.$$ Then there exists a regular function $F$ on $\B_1\times \I_{R}$
such that  $F_{|\B_1\times\B_\epsilon} =f$. \ep
\begin{proof} The function $f$ is jointly regular in the strip $\B_1\times\B_\epsilon$. Moreover, we have by hypothesis that for every fixed $p\in\B_1$ the function $f(p,q)$ is regular in $\I_{R}$ as a function of the variable $q$. We wish to conclude that $f$ is regular in $\B_1\times \I_{R}$. This has already been achieved in the proof of Theorem \ref{main}. Recall thet a key role in this step was played by Lelong's Theorem \ref{Ll}. 
\end{proof}

\bt\label{hangestreves} Let $\Omega\subset \H^2$ be a domain with $C^2$ boundary $\partial\Omega$, and assume that $\gamma:=\B_1\times \{0\}$ is contained in $\p \Omega$. Let $U$ be a neighborhood of $(1,0)$ in $\H^2$,
and suppose that $f$ is a regular function on $\Omega\cup U$. Then there exists a neighborhood $W$ of $\gamma$ in $\H^2$ and a regular function $F$ on $W$ such that $F_{|W\cap \Omega}=f_{|W\cap \Omega}$.
That is, there exists a regular extension of $f$ to a neighborhood of $\gamma$.
\et
\begin{proof}
Let $\Omega$ be locally defined by $\rho(p,q)<0$. We can assume without loss of generality that $$\rho(p,q)=\rho(x_0,\dots,x_3,y_0,\dots,y_3)=y_0+o(x^2,y^2).$$ 
Moreover, we can assume that $\nabla\rho\,(p,0)=(0,1+\sigma (p))$, where $|\sigma (p)|\le \epsilon ' |p|$ for $\epsilon '$ small.
For $\epsilon >0$ consider the translated ball $\B_1\times \{ -\epsilon\}$. Note that  for some $\epsilon'>0$ we have that $\B_1\times \B_{\epsilon (1-\epsilon ')}(-\epsilon)$ is contained in $\Omega$. We complexify $f$ with respect to both variables $p$ and $q$ and consider the Taylor series with respect to $q$ at $(0,-\epsilon)$ of the complexified function $\tilde{f}$. Then
\begin{equation}\label{seriec} \tilde{f}(p^\C,q^\C) =\sum^{+\infty}_{|\alpha |=1}\frac{\partial^{\alpha}_q \tilde{f}(p^\C,-\epsilon)}{\alpha !}(q^{\C}-\epsilon)^\alpha. \end{equation}
Here $p^\C$ and $q^{\C}$ denote the complexified variables in the sense already described in the proof of Proposition \ref{main}.
We now consider the complex disc $D$ parameterized by $\tau\to (\tau,0,0,0,-\epsilon,0,0,0)$, $|\tau|\le 1-\epsilon''$. For some small value of $\epsilon''$, the disc $D$ lies entirely inside the domain of the complexified function $\tilde{f}$. For simplicity, we rescale to make the radius of $D$ equal to 1. We now consider the restriction to $D$ of the coefficients of the series \eqref{seriec}. We let
$$ u_\alpha (\tau):= \log\left|\frac{\partial^{\alpha}_q \tilde{f}(\tau,-\epsilon)}{\alpha !}\right|^{\frac{1}{|\alpha|}}.$$
The functions $u_\alpha$ are subharmonic and uniformly bounded from above by Remark \ref{rem} and equation \eqref{bdalfa}.
By the submean property we have 
$$ u_{\alpha}(0) \le \frac {1}{2\pi} \int_0^{2\pi} u_{\alpha}(e^{i\theta})d\theta. $$
Fatou's lemma yields
\begin{equation}\label{F} \limsup_{|\alpha|\to +\infty} u_{\alpha}(0) \le \frac {1}{2\pi} \int_0^{2\pi} \limsup_{|\alpha| \to +\infty } u_{\alpha}(e^{i\theta})d\theta.
\end{equation}
 By the convergence of \eqref{seriec} we have that, up to taking a slightly smaller $\epsilon$, 
 
 \begin{equation}\label{43}
 \limsup_{|\alpha| \to +\infty } u_{\alpha}(e^{i\theta})\leq -\log((\sqrt{2}-1) \epsilon)\quad  \forall \theta.
 \end{equation}
  By hypothesis, the function $f$ is regular on the product of two small balls of radius $2\delta$ centered at the point $(1,-\epsilon)$. It is therefore possible to complexify $f$ on the product $\B^{\C}_{ \delta }(1) \times \B^{\C}_{ \delta }(-\epsilon)$. Hence, when $\theta$ is close to $0$, the series in \eqref{seriec} evaluated at $p^{\C}=e^{i\theta}$ converges on a ball of radius at least $\delta$. For these values of $\theta$, the estimate \eqref{43} can be therefore improved to 
   \begin{equation*}
 \limsup_{|\alpha| \to +\infty } u_{\alpha}(e^{i\theta})\leq -\log(\delta).
 \end{equation*}
 We obtain
\begin{equation*}
\begin{split}
 \limsup_{|\alpha| \to +\infty } u_\alpha(0)&\le \frac 1{2\pi} \Big(-(2\pi -\delta) \log((\sqrt{2}-1)\epsilon) - \delta \log((\sqrt{2}-1)\delta)\Big)\\&=\frac{\delta}{2\pi}\log\frac \epsilon\delta -\log((\sqrt{2}-1)\epsilon). 
 \end{split}
 \end{equation*}
By taking the exponential we have 
$$  \limsup_{|\alpha| \to +\infty } \left|\frac{\partial^{\alpha}_q \tilde{f}(0,-\epsilon)}{\alpha !}\right|^{\frac{1}{|\alpha|}} \le \frac {1}{(\sqrt{2}-1)\epsilon} \left( \frac\epsilon\delta\right)^{\frac \delta{2\pi}},$$
which implies that $r_q(0,-\epsilon) \ge \epsilon(\sqrt{2}-1)  \left( \frac\delta\epsilon\right)^{\frac \delta{2\pi}} $. This estimate holds true for all $p$ close to $0$.
By choosing $\epsilon$ small enough, the radius of convegence is uniformly greater than $\epsilon$, and therefore the series defining $f$ converges on a neighborhood of $(0,0)$, thus extending $f$ past the boundary by Proposition \ref{proposizione4}.
\end{proof}

\section{Acknowledgements}
The authors would like to thank the anonymous referee who pointed out a mistake in the first version of this paper. We also acknowledge useful comments from a second referee. Their suggestions helped to improve the exposition and the accuracy of the paper.


\begin{thebibliography}{ST12}
\bibitem{A67} V. Avanissian, Sur l'harmonicit\'e des fonctions s\'epar\'ement harmoniques. (French) {\em S\'eminaire de Probabilit\'es (Univ. Strasbourg, Strasbourg, 1966/67), Vol. I}, pp. 3--17, {\em Springer, Berlin.}
\bibitem{B05} L. Baracco, Extension of holomorphic functions from one side of a hypersurface. {\em Canad. Math. Bull.} {\bf 48} (2005), no. 4, 500--504.
 \bibitem{CSSS} F. Colombo, I. Sabadini, F. Sommen, and D. C. Struppa, (2004) {\it Analysis of Dirac systems and computational algebra}. Progress in Mathematical Physics, {\bf 39}. {\em Birkh\"auser Boston, Inc., Boston, MA,} 2004.
\bibitem{D} C. A. Deavours, The quaternion calculus. {\em Amer. Math. Monthly} {\bf 80} (1973), 995--1008. 
\bibitem{F1} R. Fueter, Die Funktionentheorie der Differentialgleichungen $\Delta u=0$ und $\Delta\Delta u=0$ mit vier reellen Variablen. {\em Comment. Math. Helv.}, {\bf 7} (1935), 307--330.
\bibitem{F2} R. Fueter, \"Uber die analytische Darstellung der regul\"aren Funktionen einer Quaternionen variablen. {\em Comment. Math. Helv.}, {\bf 3} (1936), 371--378.
\bibitem{GHS} K. G\"urlebeck, K. Habetha and W. Spr\"o{\ss}ig, Application of holomorphic functions in two and higher dimensions. {\em Birkhäuser/Springer, [Cham],} 2016.
\bibitem{Ha} H. Haefeli, Hyperkomplexe Differentiale, {\em Comment. Math. Helv.}, {\bf 20} (1947), 382--420.
\bibitem{HT} N. Hanges and F. Tr\`eves, Propagation of holomorphic extendability of CR functions. {\em Math. Ann.} {\bf 263} (1983), no. 2, 157--177. 
 \bibitem{H} L. H\"ormander, An introduction to complex analysis in several variables. Third edition. North-Holland Mathematical Library, 7. {\em North-Holland Publishing Co., Amsterdam}, 1990.
\bibitem{L} P. Lelong,  Fonctions plurisousharmoniques et fonctions analytiques de variables r\'eelles. (French) {\em Ann. Inst. Fourier (Grenoble)} {\bf 11} (1961), 515--562.
\bibitem{MSZ} R. Manfrin, A. Scalari, and G. Zampieri, Propagation along complex curves on a hypersurface. {\em Kyushu J. Math.} {\bf 52} (1998), no. 1, 15--22.
\bibitem{M31} Gr. C. Moisil, Sur les quaternions monog\`enes, {\em Bull. Sci. Math. (Paris)} {\bf LV} (1931), 68--74.
\bibitem{P} D. Pertici, Regular functions of several quaternionic variables. (Italian) {\em Ann. Mat. Pura Appl.} {\bf 4} 151 (1988), 39--65.
\bibitem{S} A. Sudbery, Quaternionic analysis. {\em Math. Proc. Cambridge Philos. Soc.} {\bf 85} (1979), no. 2, 199--224.
\bibitem{Z} G. Zampieri, Complex analysis and CR geometry. University Lecture Series, 43. {\em American Mathematical Society, Providence, RI,} 2008.
 

 							
 							\end{thebibliography}
\end{document}